\documentclass[a4paper,10pt]{article}

\usepackage{authblk}
\usepackage{a4wide}
\usepackage[utf8]{inputenc}
\usepackage{amsmath}
\usepackage{amsthm}
\usepackage{amssymb}
\usepackage{color}
\usepackage[english]{babel}

\newcommand{\scp}[2]{\left\langle#1,#2\right\rangle}

\newcommand{\R}{\mathbb{R}}

\newcommand{\inv}{^{-1}}

\DeclareMathOperator{\supp}{supp}
\DeclareMathOperator{\curl}{curl}
\DeclareMathOperator{\Curl}{Curl}

\DeclareMathOperator{\Div}{Div}

\DeclareMathOperator{\so}{\mathfrak{so}}
\DeclareMathOperator{\sym}{sym}
\renewcommand{\skew}{\mathrm{skew}}
\DeclareMathOperator{\mat}{mat}
\DeclareMathOperator{\smat}{axl\inv}

\DeclareMathOperator{\skewvec}{skewvec}
\DeclareMathOperator{\symvec}{symvec}
\DeclareMathOperator{\dvec}{diagvec}
\newcommand{\na}{\nabla}
\DeclareMathOperator{\vect}{vec}
\DeclareMathOperator{\axl}{axl}

\DeclareMathOperator{\diag}{diag}

\newcommand{\sodrei}{{\rm SO}(3)}

\newtheorem{theorem}{Theorem}[section]
\newtheorem{remark}[theorem]{Remark}

\newtheorem{lemma}[theorem]{Lemma}

\newtheorem{conjecture}[theorem]{Conjecture}
\newcommand{\Fpneu}{P}

\newcommand{\ol}{\overline}
\newcommand{\ub}{\underbrace}
\newcommand{\om}{\Omega}
\newcommand{\omq}{\ol{\om}}
\newcommand{\ga}{\Gamma}
\newcommand{\p}{\partial}
\newcommand{\rt}{\R^3}
\newcommand{\rni}{\R^9}
\newcommand{\rN}{\R^N}
\newcommand{\rtt}{\R^{3\times3}}
\newcommand{\rttt}{\R^{(3\times3)\times3}}
\newcommand{\rnini}{\R^{9\times9}}
\newcommand{\rNN}{\R^{N\times N}}
\newcommand{\rNNN}{\R^{(N\times N)\times N}}
\newcommand{\equi}{\Leftrightarrow}
\newcommand{\impl}{\Rightarrow}
\renewcommand{\iff}{\equi}
\newcommand{\nah}{\hat{\na}}

\newcommand{\Leb}{\mathsf{L}}
\newcommand{\Sob}{\mathsf{H}}
\newcommand{\SobW}{\mathsf{W}}
\newcommand{\Con}{\mathsf{C}}
\renewcommand{\L}[2]{\Leb^{#1}_{#2}}
\renewcommand{\H}[2]{\Sob^{#1}_{#2}}
\newcommand{\W}[2]{\SobW^{#1}_{#2}}
\newcommand{\C}[2]{\Con^{#1}_{#2}}

\newcommand{\Lo}{\L{1}{}}
\newcommand{\Lt}{\L{2}{}}
\newcommand{\Lp}{\L{p}{}}
\newcommand{\Lq}{\L{q}{}}
\newcommand{\Lr}{\L{r}{}}
\newcommand{\Ls}{\L{s}{}}
\newcommand{\Li}{\L{\infty}{}}
\newcommand{\Loom}{\Lo(\om)}
\newcommand{\Ltom}{\Lt(\om)}
\newcommand{\Lpom}{\Lp(\om)}
\newcommand{\Lqom}{\Lq(\om)}

\newcommand{\Liom}{\Li(\om)}
\newcommand{\Loomrnnn}{\Lo(\om;\rNNN)}
\newcommand{\Ltomrnnn}{\Lt(\om;\rNNN)}
\newcommand{\Liomrtt}{\Li(\om;\rtt)}

\newcommand{\Lpomrtt}{\Lp(\om;\rtt)}
\newcommand{\Lqomrtt}{\Lq(\om;\rtt)}

\newcommand{\Lsomrtt}{\Ls(\om;\rtt)}

\newcommand{\Woo}{\W{1,1}{}}
\newcommand{\Wop}{\W{1,p}{}}
\newcommand{\Woq}{\W{1,q}{}}

\newcommand{\Woqc}{\W{1,q}{\circ}}
\newcommand{\Wor}{\W{1,r}{}}
\newcommand{\Wos}{\W{1,s}{}}
\newcommand{\Woi}{\W{1,\infty}{}}

\newcommand{\Wtq}{\W{2,q}{}}

\newcommand{\Wooom}{\Woo(\om)}
\newcommand{\Wopom}{\Wop(\om)}
\newcommand{\Woqom}{\Woq(\om)}

\newcommand{\Woqcom}{\Woqc(\om)}
\newcommand{\Worom}{\Wor(\om)}

\newcommand{\Wooomrn}{\Woo(\om;\rN)}

\newcommand{\Wopomrt}{\Wop(\om;\rt)}

\newcommand{\Woromrtt}{\Wor(\om;\rtt)}
\newcommand{\Wosomrtt}{\Wos(\om;\rtt)}
\newcommand{\Woiomrn}{\Woi(\om;\rN)}
\newcommand{\Woiomrt}{\Woi(\om;\rt)}

\newcommand{\Wtqomrt}{\Wtq(\om;\rt)}

\newcommand{\Ho}{\H{1}{}}
\newcommand{\Hoc}{\H{1}{\circ}}
\newcommand{\Ht}{\H{2}{}}
\newcommand{\Hoom}{\Ho(\om)}

\newcommand{\Hoomrt}{\Ho(\om;\rt)}
\newcommand{\Htomrt}{\Ht(\om;\rt)}
\newcommand{\HoomrN}{\Ho(\om;\rN)}

\newcommand{\Hocomga}{\Hoc(\om,\ga)}
\newcommand{\Hocomgart}{\Hoc(\om,\ga;\rt)}

\newcommand{\Ci}{\C{\infty}{}}
\newcommand{\Cic}{\C{\infty}{\circ}}
\newcommand{\Co}{\C{1}{}}

\newcommand{\Cz}{\C{0}{}}

\newcommand{\Cicomrtt}{\Cic(\om;\rtt)}
\newcommand{\Cicomga}{\Cic(\om,\ga)}
\newcommand{\Cicomgart}{\Cic(\om,\ga;\rt)}
\newcommand{\Coomqrtt}{\Co(\omq;\rtt)}
\newcommand{\Ciomq}{\Ci(\omq)}
\newcommand{\Czomq}{\Cz(\omq)}

\newcommand{\Ciomqrt}{\Ci(\omq;\rt)}

\newcommand{\WC}[2]{\SobW^{#1}(\Curl,#2)}

\newcommand{\WoComrtt}{\WC{1}{\om;\rtt}}

\newcommand{\WsCom}{\WC{s}{\om}}
\newcommand{\WsComrtt}{\WC{s}{\om;\rtt}}

\newcommand{\normtwo}[2]{|\!|#1|\!|_{#2}}
\newcommand{\normthree}[2]{|\!|\!|#1|\!|\!|_{#2}}
\newcommand{\norm}[1]{\normtwo{#1}{}}
\newcommand{\normp}[1]{\normthree{#1}{}}
\newcommand{\normLoom}[1]{\normtwo{#1}{\Loom}}
\newcommand{\normLtom}[1]{\normtwo{#1}{\Ltom}}
\newcommand{\normHoom}[1]{\normtwo{#1}{\Hoom}}

\newcommand{\scpLtom}[2]{\scp{#1}{#2}_{\Ltom}}

\begin{document}

\title{Uniqueness of integrable solutions to $\na\zeta=G\zeta$, $\zeta|_{\ga}=0$ 
for integrable tensor-coefficients $G$ and applications to elasticity}
\author[1]{Johannes Lankeit}
\author[1,2]{Patrizio Neff}
\author[1]{Dirk Pauly}
\affil[1]{Universit\"at Duisburg-Essen, Fakult\"at f\"ur Mathematik}
\footnotetext[2]{To whom correspondence should be addressed. e-mail: {\tt patrizio.neff@uni-due.de}}
\date{\today}
\maketitle

\begin{abstract}
Let $\om\subset\rN$ be a Lipschitz domain
and $\ga$ be a relatively open and non-empty 
subset of its boundary $\p\om$.
We show that the solution to the linear first order system 
\begin{align}
\label{eq:dglabstract}
\na\zeta=G\zeta,\quad\zeta|_{\ga}=0
\end{align}
vanishes if $G\in\Loomrnnn$ and $\zeta\in\Wooomrn$, 
which is the case e.g. for square integrable solutions $\zeta$ of \eqref{eq:dglabstract} 
and $G\in\Ltomrnnn$. As a consequence, we prove 
$$\normp{\cdot}:\Cicomgart\to[0,\infty),\quad
u\mapsto\normLtom{\sym(\na u\Fpneu\inv)}$$ 
to be a norm for $\Fpneu\in\Liomrtt$ with $\Curl\Fpneu\in\Lpomrtt$, 
$\Curl\Fpneu\inv\in\Lqomrtt$ for some $p,q>1$ 
with $1/p+1/q=1$ as well as $\det\Fpneu\geq c^+>0$.
We also give a new and different proof for the so called 
`infinitesimal rigid displacement lemma' in curvilinear coordinates: 
Let $\Phi\in\Hoomrt$ satisfy $\sym(\na\Phi^\top\na\Psi)=0$ 
for some $\Psi\in\Woiomrt\cap\Htomrt$ with
$\det\na\Psi\geq c^+>0$. Then there exists a constant translation vector 
$a\in\rt$ and a constant skew-symmetric matrix 
$A\in\so(3)$, such that $\Phi=A\Psi+a$. 

\vspace{1cm}\noindent
\end{abstract}

\section{Introduction}

Consider the linear first order system of partial differential equations
\begin{align}
\label{gleichung}
\na\zeta=G\,\zeta,\quad\zeta|_{\ga}=0.
\end{align}
Obviously, one solution is $\zeta=0$. But is this solution unique? 
The answer is not as obvious as it may seem; consider for example 
in dimension $N:=1$, $G(t):=1/t$ 
in the domain $\om:=(0,1)$ with $\ga:=\{0\}\subset\p\om$. Then 
$\zeta:=\mathrm{id}\neq0$ solves \eqref{gleichung}.
However, in the latter example the solution becomes unique if $G\in\Loom$,
which is easily deduced from Gronwall's lemma. 
Here we can see that we will need integrability conditions 
on the coefficient $G$; 
for a precise formulation of the result see section \ref{sec:results}. 
The uniqueness of the solution to \eqref{gleichung} makes 
\begin{align}
\label{eq:normdefIntroduction}
\normp{u}:=\normLtom{\sym(\na u\Fpneu\inv)}
\end{align}
a norm on 
\begin{align*}
\Cicomgart
&:=\{u\in\Ciomqrt\,:\,{\rm dist}(\supp u,\ga)>0\},\\
\Ciomqrt
&:=\{u|_{\om}\,:\,u\in\Cic(\rt;\rt)\}
\end{align*}
for $\Fpneu\in\Liomrtt$ with
$\det\Fpneu\geq c^+>0$ if 
$\Curl\Fpneu\in\Lpomrtt$, 
$\Curl\Fpneu\inv\in\Lqomrtt$ for some $p,q>1$ and $1/q+1/p=1$.
Here the $\Curl$ of a matrix field is defined as the row-wise standard $\curl$ in $\rt$.

The question whether an expression of the form \eqref{eq:normdefIntroduction} 
is a norm arises when trying to generalize Korn's first inequality 
to hold for non-constant coefficients, i.e.,
\begin{align}
\label{eq:introkornnonconst}
\exists\,c>0\quad\forall u\in\Hocomgart\qquad
\normLtom{\sym(\na u\Fpneu\inv)}\geq c\normHoom{u},
\end{align}
which was first done for $\Fpneu,\Fpneu\inv,\Curl\Fpneu\in\Coomqrtt$ 
by Neff in \cite{Neff00b}, cf. \cite{Pompe03}.
Here $\Hocomgart$ denotes the closure of
$\Cicomgart$
in $\Hoomrt$. 
The classical Korn's first inequality
is obtained for $\Fpneu$ being the identity matrix, see 
 \cite{Korn09,Ciarlet10,Neff00b,Neff_Pauly_Witsch_Korn_diff_forms_m2as12,Neff_Pauly_Witsch_cracad11,Neff_Pauly_Witsch_arma12}. 
The inequality \eqref{eq:introkornnonconst}  
has been proved in \cite{Pompe03} to hold for continuous $\Fpneu\inv$, 
whereas it can be violated for $\Fpneu\inv\in\Liom$ 
or $\Fpneu\inv\in SO(3)$ a.e.. 
The counterexamples, given by Pompe in \cite{Pompe03} and \cite{Pompe10}, 
see also \cite{Neff_Pompe11},
each use the fact that for such $\Fpneu$ 
an expression of the form of $\normp{\cdot}$ is not a norm 
(It has a nontrivial kernel) on the spaces of functions considered.
Quadratic forms of the type \eqref{eq:introkornnonconst} 
arise in applications to geometrically exact models of shells, 
plates and membranes, in micromorphic and Cosserat type models and in plasticity, 
\cite{Klawonn_Neff_Rheinbach_vanis09,Neff_membrane_existence03,Neff_micromorphic_rse_05,Neff_plate05_poly,Neff01c}.

The so called `infinitesimal rigid displacement lemma in curvilinear coordinates', 
a version of which can be found in \cite{Anicic_Dret_Raoult2004} 
and which is important for linear elasticity in curvilinear coordinates 
(see also \cite{Ciarlet99,CiarletMardare04}) states the following: 
If $\om\subset\rN$ is a bounded domain, 
$\Psi\in\Woiomrn$ satisfying $\det\na\Psi\geq c^+>0$ a.e. 
and $\Phi\in\HoomrN$ with $\sym(\na\Phi^\top\na\Psi)=0$ a.e.,
then on a dense open subset of $\om$ there exist locally constant mappings 
$a:\om\to\rN$ and $A:\om\to\so(N)$ 
such that locally $\Phi=A\Psi+a$.
If $\om$ is Lipschitz then the terms `locally' can be dropped.
In their proof \cite{Anicic_Dret_Raoult2004}, 
the authors apply the chain rule to $\Theta:=\Phi\circ\Psi\inv$ 
and use the observation that the conditions
$\sym(\na\Phi^\top\na\Psi)=0$ and $\sym(\na\Phi(\na\Psi)\inv)=0$ are equivalent
by a clever conjugation with $(\na\Psi)\inv$, this is 
\begin{align}
\label{symconj}
(\na\Psi)^{-\top}\sym(\na\Phi^\top\na\Psi)(\na\Psi)\inv
=\sym(\na\Phi(\na\Psi)\inv)
=\sym(\na(\Phi\circ\Psi\inv))\circ\Psi,
\end{align}
together with the classical infinitesimal rigid displacement lemma applied on 
$\Theta$, defined on the domain $\Psi(\om)$.
If to this lemma a boundary condition $\Phi=0$ on a relatively open subset 
of the boundary is added, 
one obtains $\Phi=0$ (cf. \cite[1.7-3(b)]{Ciarlet99}).

The main part of our proof for $\normp{\cdot}$ being a norm is also concerned 
with obtaining $u=0$ from $\sym(\na u\Fpneu\inv)=0$. 
By taking $\Fpneu=\na\Psi$ to be a gradient, 
we present another proof of the infinitesimal rigid displacement lemma 
in dimension $N=3$ which yields 
$\Phi=A\Psi+a$ with $A\in\so(N)$, $a\in\rN$.
We need slightly more regularity 
but do not use the chain rule for $\Theta$.

The key tool for obtaining our results
is Neff's formula for the $\Curl$ of the product of two matrices, 
the first of which is skew-symmetric. 
We state a generalization of this formula in section \ref{sec:diecurlformel}.

This paper is organized as follows:
The next section states the main results that will be proven in the subsequent chapters.
Section \ref{sec:zeroonlines} provides a tool that gives $\zeta=0$ 
on lines and is used in section \ref{sec:zerooncube} 
where this is extended to cubes. 
Section \ref{sec:fromcubetodomain} then takes care of the whole domain 
if a `$(\zeta=0)$-cube' is given as starting point. 
In section \ref{sec:proofs} the uniqueness theorem of section \ref{sec:results} is proven, 
mainly by putting together the results of the previous sections. 
After that and before applying the theorem we have a closer look 
at the formula for the $\Curl$ of a product of matrices (section \ref{sec:diecurlformel}). 
Finally, in sections \ref{sec:thenorm} and \ref{sec:infrigdispl}, respectively, 
we prove that $\normp{\cdot}$ is a norm 
and we present our new proof of the infinitesimal rigid displacement lemma.

\section{Results}
\label{sec:results}

Let us first note that by $\na$ we denote not only the gradient of a scalar-valued function, 
but also (as an usual gradient row-wise) 
the derivative or Jacobian of a vector-field.
The $\Curl$ of a matrix is to be taken row-wise as usual $\curl$ for vector fields.

\begin{theorem}[\sf Unique Continuation]
\label{result:1}
Let $\om\subset\rN$, $N\in\mathbb{N}$, be a Lipschitz domain,
$\ga$ be a relatively open and non-empty subset of $\p\om$ as well as
$G\in\Loomrnnn$. If $\zeta\in\Wooomrn$ solves
$$\na\zeta=G\,\zeta,\quad\zeta|_\ga=0,$$
then $\zeta=0$.
\end{theorem}

From the differential equation itself
it is not a priori clear that $\zeta$ belongs to $\Wooom$. 
But this can be ensured by requiring higher integrability of $G$ and $\zeta$, 
since for bounded domains, e.g.,
the conditions $G\in\Ltom$ and $\zeta\in\Ltom$ 
imply $\na\zeta\in\Loom$ and hence $\zeta\in\Wooom$, 
where an application of the theorem ensures $\zeta=0$. 
Thus we have obtained the uniqueness of $\Ltom$-solutions 
if the coefficient $G$ are square-integrable.
Of course, the same holds if $\zeta\in\Lpom$ for arbitrary $p\geq1$. 
Then $G$ at least needs to be an $\Lqom$-function, 
where $1/p+1/q=1$.

\begin{theorem}[\sf Norm]
\label{result:norm}
Let $\om\subset\rt$ be a Lipschitz domain, 
$\emptyset\neq\ga\subset\p\om$ be relatively open, 
$\Fpneu\in\Liomrtt$ with 
$\det\Fpneu\geq c^+>0$, 
$\Curl\Fpneu\in\Lpomrtt$,
$\Curl\Fpneu\inv\in\Lqomrtt$
for some $p,q>1$ with $1/p+1/q=1$. Then 
\begin{align}
\label{eq:normdef}
\normp{\cdot}:\Cicomgart\to[0,\infty),\quad
u\mapsto\normLtom{\sym(\na u\Fpneu\inv)}
\end{align}
defines a norm.
\end{theorem}

\begin{remark}
In the case of $p=q=2$ and for 
$\Fpneu\in\sodrei$ a.e., 
$\Curl\Fpneu\inv\in\Ltom$ is no additional condition, 
since then $\Curl\Fpneu\in\Ltom\iff\Curl\Fpneu\inv\in\Ltom$. 
(Note that for $\Fpneu\in\sodrei$ a.e. 
generally $\Fpneu,\Curl\Fpneu\in\Lpom$ 
is equivalent to $\Fpneu\in\Wopom$, cf. \cite{Neff_curl06}.)
\end{remark}

\begin{conjecture}
Theorem \ref{result:norm} holds for 
$\Fpneu\in\Liom$ with $\Curl\Fpneu\in\Lpom$
and $\det\Fpneu\geq c^+>0$ for some $p>1$ or even $p\geq1$.
\end{conjecture}

\begin{remark}
Since the norms $\normp{\cdot}$ and $\normHoom{\cdot}$ 
are not shown to be equivalent, it is not clear, 
whether the spaces $\Hocomga=\ol{\Cicomga}^{\normHoom{\cdot}}$ 
and $\ol{\Cicomga}^{\normp{\cdot}}$ coincide. 
However, by \cite{Pompe03}, these norms are equivalent 
if $\Fpneu\in\Czomq$ with $\det\Fpneu\geq c^+>0$. 
\end{remark}

\begin{conjecture}
The norms are equivalent if
$\Fpneu\in\Liom$ with $\Curl\Fpneu\in\Lpom$
and $\det\Fpneu\geq c^+>0$ for some $p>1$ or even $p\geq1$.
\end{conjecture}

\begin{theorem}[\sf Infinitesimal Rigid Displacement Lemma]
\label{result:infrigdisplemma}
Let $\om\subset\rt$ be a Lipschitz domain. 
Moreover, let $\Phi\in\Wopomrt$
and $\Psi\in\Woiomrt\cap\Wtqomrt$ with 
$\det\na\Psi\geq c^+>0$ a.e.
and $p,q>1$, $1/p+1/q=1$. If
$$\sym(\na\Phi^\top\na\Psi)=0$$ 
then there exist $a\in\rt$ 
and a constant skew-symmetric matrix $A\in\so(3)$, 
such that $\Phi=A\Psi+a$.
\end{theorem}

\begin{remark}
When comparing two nearby configurations of an elastic body, 
namely $\hat{\Psi}:\om\to\rt$ and $\Psi:\om\to\rt$, 
following Ciarlet \cite{CiarletMardare04} we may always write 
$\hat{\Psi}=\Psi+\Phi$, where $\Phi:\om\to\rt$ 
is the displacement from $\Psi$ to $\hat{\Psi}$.
The respective metric tensors of the two configurations are 
$\na\hat\Psi^\top\na\hat{\Psi}$ and $\na\Psi^\top\na\Psi$.
In terms of the displacement $\Phi$ 
to lowest order we have for the $\Phi$-linearized change of the metric
\begin{align*}
[\na\hat\Psi^\top\na\hat{\Psi}-\na\Psi^\top\na\Psi]_{\mathrm{lin},\Phi}
&=[(\na\Psi+\na\Phi)^\top(\na\Psi+\na\Phi)-\na\Psi^\top\na\Psi]_{\mathrm{lin},\Phi}\\
&=\na\Phi^\top\na\Psi+\na\Psi^\top\na\Phi=2\sym(\na\Phi^\top\na\Psi).
\end{align*}
Therefore, the infinitesimal rigid displacement lemma expresses the fact that 
if the linearized change of the metric is zero, 
then the displacement must be (the linearized part of) some rigid displacement.
\end{remark}

\section{Proof of the uniqueness theorem}

We start with some preliminaries.

\subsection{Vanishing in intervals}

\label{sec:zeroonlines}

Let $-\infty<a<b<\infty$ and $I:=(a,b)$.

\begin{lemma}
\label{lemma:zeroonlines}
Let $G\in\Lo(I;\rNN)$, $\zeta\in\Woo(I;\rN)$ with
$\zeta'=G\,\zeta$ and $\zeta(a)=0$. Then $\zeta=0$.
\end{lemma}

\begin{proof} We use Gronwall's inequality.
Because $\zeta\in\Woo(I)$, $\zeta$ is absolutely continuous
and hence it can be written as an integral over its derivative:
\begin{align*}
\zeta(x)=\zeta(a)+\int_a^x\zeta'(t)\,{\rm d}t
=\zeta(a)+\int_a^xG(t)\zeta(t)\,{\rm d}t
\quad\impl\quad
|\zeta(x)|\leq|\zeta(a)|+\int_a^x\norm{G(t)}|\zeta(t)|\,{\rm d}t.
\end{align*}
An application of Gronwall's inequality leads to 
$$|\zeta(x)|\leq|\zeta(a)|\exp\big(\int_a^x\norm{G(t)}\,{\rm d}t\big)=0,\quad x\in I,$$
which concludes the proof.
\end{proof}

\subsection{Vanishing in cubes}

Let $Q$ be a cuboid in $\rN$ and let $\ga$ be a face of $Q$, 
i.e., $Q=\ga\times I$ with $I$ from the previous section.
By \cite[Th. 2.1.4]{Ziemer89} we have that
for $u\in\Lo(Q)$ the following is equivalent:
$u\in\Woo(Q)$, if and only if $u$ has a representative
which is absolutely continuous on almost all line segments in $Q$ parallel to the coordinate axes 
and whose (classical a.e.) partial derivatives belong to $\Lo(Q)$.
These classical partial derivatives coincide with the weak derivatives almost everywhere.
In particular, if $u\in\Woo(Q)$ 
then $u_{\gamma}:=u(\gamma,\cdot)\in\Woo(I)$ f.a.a. $\gamma\in\ga$.
Of course, the same holds for $u\in\Wop(Q)$ with $p\geq 1$.

\label{sec:zerooncube}
\begin{lemma}
\label{lemma:zerooncube}
Let $G\in\Lo(Q;\rNNN)$ and $\zeta\in\Woo(Q;\rN)$ 
with $\na\zeta=G\zeta$ and $\zeta|_{\ga}=0$.
Then $\zeta=0$.
\end{lemma}

\begin{proof}
Since
$$\normLoom{\zeta}=\int_\ga\int_a^b|\zeta(\gamma,x)|\,{\rm d}x\,{\rm d}\gamma
=\int_\ga\int_a^b|\zeta_{\gamma}(x)|\,{\rm d}x\,{\rm d}\gamma$$
we only have to show $\zeta_{\gamma}=0$ a.e..
As $\zeta\in\Woo(Q)$, 
$\zeta_{\gamma}\in\Woo(I)$ f.a.a. $\gamma\in\ga$ 
by \cite[Th. 2.1.4]{Ziemer89}, as mentioned before. Since
$\zeta_{\gamma}'$ is the last column $\na\zeta e^{N}$ of $\na\zeta$, we have
$$\zeta_{\gamma}'=\na\zeta(\gamma,\cdot) e^{N}=G(\gamma,\cdot)\zeta(\gamma,\cdot) e^{N}
=G(\gamma,\cdot)\zeta_{\gamma}e^{N}=:G_{\gamma}\zeta_{\gamma}.$$
For fixed $(\gamma,x)\in Q$, $G(\gamma,x)$ is a linear mapping from 
$\rN$ to $\rNN$, its product with $\zeta_\gamma(x)\in\rN$ is an element of $\rNN$ 
and multiplication by $e^N$ gives an element of $\rN$ depending linearly on $\zeta_\gamma(x)$.
Hence, $G_{\gamma}(x)$ is a linear mapping from $\rN$ to $\rN$ a.e..
Even $G_{\gamma}\in\Lo(I;\rNN)$ holds, since $G\in\Lo(Q)$.
Also, $\zeta|_{\ga}=0$ implies $\zeta_{\gamma}(a)=0$ f.a.a. $\gamma\in\ga$.
By Lemma \ref{lemma:zeroonlines} we obtain $\zeta_{\gamma}=0$.
\end{proof}

\subsection{Unique continuation}
\label{sec:fromcubetodomain}

\begin{lemma}
\label{lemma:fromcubetocube}
Let $G\in\Loomrnnn$ and
$\zeta\in\Wooomrn$ with $\na\zeta=G\zeta$. 
Moreover, let $\zeta$ vanish in an open ball $B\subset\om$.
Then $\zeta=0$.
\end{lemma}

\begin{proof}
Let $\om$ be convex and pick some $x_{1}\in B$.
Then we can take a straight line between $x_1$ 
and some other point $x_2\in\om$ and a cuboid $Q$ 
containing this line and having one face being entirely located in $B$. 
Then by Lemma \ref{lemma:zerooncube} 
$\zeta=0$ in $Q$ and hence in a whole neighborhood of $x_2$. 
Since $x_{2}$ was arbitrary, we have $\zeta=0$ in $\om$.
By induction this can be carried over to connected unions of finitely many convex sets 
and hence works for path-connected sets, 
because every path between two points can be covered by such a finite union. 
Since domains are path-connected, we finally achieve $\zeta=0$ in $\om$.
\end{proof}

\begin{remark} 
Lemma \ref{lemma:fromcubetocube} can also be stated as: 
The equation $\na\zeta=G\zeta$, i.e.,
the operator $\na-G$ has the unique continuation property.
Moreover, it is enough that $\zeta$ vanishes on a small part of some
$(N-1)$-dimensional hyper-plane.
\end{remark}

\subsection{Proof of Theorem \ref{result:1}}

\label{sec:proofs}

Let $\zeta$ be as in Theorem \ref{result:1}.
If we can show that $\zeta$ vanishes on an open set, 
we can apply Lemma \ref{lemma:fromcubetocube} and hence, $\zeta$ must vanish in the whole of $\om$.
To make $\zeta=0$ on an open set, 
we transform a part of $\ga$, where we know $\zeta$ to be zero, 
and a neighborhood $U$ onto a cuboid $Q$, 
where we can use Lemma \ref{lemma:zerooncube} 
and the transformed function is forced to vanish, 
hence also $\zeta$ must vanish on $U$. 
Let us pick a point on $\ga$ and a corresponding open neighborhood
$U$ as well as a bijective bi-Lipschitz transformation 
$\varphi:\hat{Q}:=(-1,1)^N\to U$,
mapping the cuboid $Q:=(-1,1)^{N-1}\times(0,1)$ 
onto $U\cap\om$
and $(-1,1)^{N-1}\times\{0\}$ onto $U\cap\ga$.
Now $\tilde{G}:=G\circ\varphi\in\Lo(Q;\rNNN)$
and, see again e.g. \cite[Th. 2.2.2]{Ziemer89},
$\tilde{\zeta}:=\zeta\circ\varphi\in\Woo(Q;\rN)$.
By the chain rule we have
\begin{align*}
\na\tilde\zeta
=((\na\zeta)\circ\varphi)\na\varphi
=((G\,\zeta)\circ\varphi)\,\na\varphi
=\tilde{G}\tilde\zeta\,\na\varphi
=:\hat{G}\tilde\zeta.
\end{align*}
Since $\na\varphi$ is uniformly bounded we get
$$\forall\,z\in Q,y\in\rN\quad\norm{\hat{G}(z)y}\leq
\norm{\tilde{G}(z)}\,|y|\,\norm{\na\varphi(z)}
\leq c\norm{\tilde{G}(z)}\,|y|
\quad\impl\quad
\norm{\hat{G}(z)}\leq c\norm{\tilde{G}(z)}$$
and hence
$$\int_Q\norm{\hat{G}(z)}{\rm d}z
\leq c\int_Q\norm{\tilde{G}(z)}{\rm d}z
=c\int_{\varphi(Q)}\norm{G(x)}\,|\det\na\varphi\inv(x)|{\rm d}x
\leq c\int_{\om\cap U}\norm{G(x)}{\rm d}x<\infty$$
since $\det\na\varphi\inv$ is uniformly bounded as well.
Thus $\hat{G}\in\Lo(Q;\rNNN)$. 
Because $\zeta$ vanishes on $U\cap\ga$, 
$\tilde\zeta$ vanishes on $F:=\varphi\inv(U\cap\ga)=(-1,1)^{N-1}\times\{0\}$.
Hence $\tilde\zeta\in\Woo(Q;\rN)$ solves
$$\na\tilde\zeta=\hat G\,\tilde\zeta,\quad
\tilde\zeta|_{F}=0.$$
Lemma \ref{lemma:zerooncube} implies $\tilde\zeta=0$ in $Q$.
Thus $\zeta=\tilde\zeta\circ\varphi\inv=0$ in $U\cap\om$, which contains an open ball.
\hfill$\square$

\section{Proof of the norm property}

\subsection{Curl of matrix-products}
\label{sec:diecurlformel}

We identify $\rtt$ and $\rni$ by the following isomorphisms:
$$\mat:\rni\to\rtt,\quad 
\begin{bmatrix}a_1\\\vdots\\a_9\end{bmatrix}\mapsto 
\begin{bmatrix}a_1&a_2&a_3\\a_4&a_5&a_6\\a_7&a_8&a_9\end{bmatrix},\quad
\vect:=\mat\inv:\rtt\to\rni$$
We also use the following canonical isomorphism to identify $\rt$ and $\so(3)$:
\begin{align*}
\axl&:\so(3)\to\R^{3},&
\begin{bmatrix}
0&-a_3&a_2\\a_3&0&-a_1\\-a_2&a_1&0
\end{bmatrix}
&\mapsto
\begin{bmatrix}
a_1\\a_{2}\\a_3
\end{bmatrix}
\intertext{Moreover, we define}
\dvec&:\rtt\to\rt,&
\begin{bmatrix}
\textcircled{$a_{1}$}&a_2&a_3\\a_4&\textcircled{$a_{5}$}&a_6\\a_7&a_8&\textcircled{$a_{9}$}
\end{bmatrix}
&\mapsto 
\begin{bmatrix}
a_1\\a_{5}\\a_9
\end{bmatrix},\\
\skewvec&:\rtt\to\rt,&
\begin{bmatrix}
a_1&\textcircled{$a_{2}$}&\textcircled{$a_{3}$}\\a_4&a_5&\textcircled{$a_{6}$}\\a_7&a_8&a_9
\end{bmatrix}
&\mapsto 
\begin{bmatrix}
-a_6\\a_{3}\\-a_2
\end{bmatrix},\\
\symvec&:\rtt\to\rt,&
\begin{bmatrix}
a_1&a_2&a_3\\\textcircled{$a_{4}$}&a_5&a_6\\\textcircled{$a_{7}$}&\textcircled{$a_{8}$}&a_9
\end{bmatrix}
&\mapsto 
\begin{bmatrix}
a_8\\-a_{7}\\a_4
\end{bmatrix}.
\end{align*}
We note $\skewvec=\symvec=\axl$, $\dvec=0$ on $\so(3)$.
Furthermore, $Ax=\axl(A)\times x$ and
$\smat(a)x=a\times x$ holds for all $A\in\so(3)$ and all $a,x\in\rt$, 
where $\times$ denotes the cross-product.

For a matrix $Y\in\rtt$ with $Y^{\top}=[y_{1}\,y_{2}\,y_{3}]$ 
and vectors $y_{n}\in\rt$ we define
\begin{align*}
L_{\diag,Y}
&=
-\begin{bmatrix}
\smat y_{1}&0&0\\
0&\smat y_{2}&0\\
0&0&\smat y_{3}
\end{bmatrix},\\
L_{\skew,Y}
&=
\begin{bmatrix}
0&-\smat y_{3}&\smat y_{2}\\
\smat y_{3}&0&0\\
0&0&0
\end{bmatrix},\\
L_{\sym,Y}
&=
\begin{bmatrix}
0&0&0\\
0&0&-\smat y_{1}\\
-\smat y_{2}&\smat y_{1}&0
\end{bmatrix},\\
L_Y:=L_{\skew,Y}+L_{\sym,Y}
&=
\begin{bmatrix}
0&-\smat y_{3}&\smat y_{2}\\
\smat y_{3}&0&-\smat y_{1}\\
-\smat y_{2}&\smat y_{1}&0
\end{bmatrix}
\end{align*}
and note $L_{Y}^{\top}=L_{Y}$.
Furthermore, for vector fields $v$ in $\rt$ we set
$$\nah v:=\vect\na v,$$
denoting the vector-field containing the nine partial derivatives 
of the three components of $v$.

Now, we extend Neff's formula from \cite[Lemma 3.7]{Neff00b} 
in two ways, such that 
it can be applied with weaker differentiability 
and for general matrices.
For this, we define
$$\WsComrtt:=\{Y\in\Lsomrtt:\Curl Y\in\Lsomrtt\}.$$

\begin{remark}
\label{skewhorem}
For skew-symmetric matrix fields we have $\WsComrtt=\Wosomrtt$,
since in this case the $\Curl$ controls all the derivatives, see \cite{Neff_curl06}.
\end{remark}

\begin{lemma}
\label{neffformula}
Let $r,s\in(1,\infty)$ with $1/r+1/s=1$.
Moreover, let
$X\in\Woromrtt$ and
$Y\in\WsComrtt$.
Then 
$XY\in\WoComrtt$
and
\begin{align}
\label{eq:curlformel_allg}
\Curl(XY)&=\mat(L_{\diag,Y}\nah\dvec+L_{\skew,Y}\nah\skewvec+L_{\sym,Y}\nah\symvec)X+X\Curl Y
\intertext{with $L_{\diag,Y},L_{\skew,Y},L_{\sym,Y}\in\Ls(\om;\rnini)$.
For skew-symmetric $X$ formula \eqref{eq:curlformel_allg} turns to}
\label{eq:curlformelschwach}
\Curl(XY)&=\mat L_Y(\nah\axl X)+X\Curl Y,
\end{align}
where $\det L_{Y}=-2(\det Y)^3$.
Hence, if $Y$is invertible, so is $L_{Y}$.
\end{lemma}

We note that for smooth ($\Co$) matrices $X,Y$, 
where $X$ is skew-symmetric, formula \eqref{eq:curlformelschwach}
was already shown in \cite[Lemma 3.7]{Neff00b}.

\begin{proof}
Since $\Ciomq$ is dense in both $\Worom$ and $\WsCom$
we have to show \eqref{eq:curlformel_allg} only for smooth matrix fields.
But this is a straight forward calculation, which we present in the appendix.
\eqref{eq:curlformelschwach} is a simple consequence from \eqref{eq:curlformel_allg}
and the assertion about the determinants has been proved already in \cite[Lemma 3.7]{Neff00b}.
\end{proof}

\subsection{Proof of Theorem \ref{result:norm}}
\label{sec:thenorm}

Let $u\in\Cicomgart$ with $\normp{u}=0$. We have to show $u=0$.
Note that $\sym(\na u\Fpneu\inv)=0$ implies
\begin{align}
\label{eq:defnA}
\na u\Fpneu\inv=A,
\end{align}
where $A$ is some skew-symmetric matrix field. 
Moreover, since $A\Fpneu=\na u$ we have
\begin{align}
\label{eq:curlnull}
\Curl(A\Fpneu)=0.
\end{align}

Without loss of generality we assume $\ga$ to be bounded
(otherwise, replace $\ga$ by a bounded open subset of itself)
and that the compact set $\supp u$ and $\ga$ are both contained in some open ball $B$.
Define $\tilde{\om}:=\om\cap B$. 
Then $\na u,\Fpneu,\Fpneu\inv$ and $A$ belong to 
$\Li(\tilde{\om},\rtt)\subset\Lr(\tilde{\om},\rtt)$
for all $r\in[1,\infty]$. 

Since $\Curl\Fpneu\inv\in\Lq(\tilde{\om},\rtt)$,
Lemma \ref{neffformula} with $X:=\na u$ and $Y:=\Fpneu\inv$
together with Remark \ref{skewhorem} show
$A\in\Woo(\tilde{\om},\rtt)$
and by \eqref{eq:curlformel_allg} even
$A\in\Woq(\tilde{\om},\rtt)$ holds.
Another application of Lemma \ref{neffformula} with $X:=A$ and $Y:=\Fpneu$ 
gives by \eqref{eq:curlformelschwach} and \eqref{eq:curlnull}
$$\mat L_{\Fpneu}(\nah\axl A)+A\Curl\Fpneu=0,$$
since $A$ is skew-symmetric. Thus, 
$\zeta:=\axl A\in\Woq(\tilde{\om},\rt)\subset\Woo(\tilde{\om},\rt)$ solves
\begin{align}
\label{zetaCurlFp}
\na\zeta=-\mat L_{\Fpneu}\inv\vect\big(\smat\zeta\Curl\Fpneu\big)=:G_{\Fpneu}\zeta.
\end{align}
Since $L_{\Fpneu},L_{\Fpneu}\inv\in\Li(\tilde{\om},\rnini)$ 
and $\Curl\Fpneu\in\Lp(\tilde{\om},\rtt)\subset\Lo(\tilde{\om},\rtt)$,
also $G_{\Fpneu}$ belongs to $\Lp(\tilde{\om},\rttt)\subset\Lo(\tilde{\om},\rttt)$.
Additionally, $A$ and hence $\zeta$ vanish on $\ga$ by \eqref{eq:defnA} since $u$ does.
By Theorem \ref{result:1}, $\zeta$ and therefore $A$ and $\na u$
vanish in $\tilde{\om}$. Thus $u=0$ in $\tilde{\om}$ because $u$ vanishes on $\ga$.
Since $\supp u\subset\tilde{\om}$ we finally obtain $u=0$ in $\om$.
\hfill$\square$

\section{Proof of Theorem \ref{result:infrigdisplemma}}
\label{sec:infrigdispl}

$A:=\na\Phi(\na\Psi)\inv\in\Lpomrtt$ is skew-symmetric by \eqref{symconj}.
Since the standard mollification preserves skew-symmetry we can pick a sequence
$(A_{n})\subset\Cicomrtt$ of skew-symmetric smooth matrices approximating
$A$ in $\Lpom$. Applying Lemma \ref{neffformula}, i.e., 
\eqref{eq:curlformelschwach}, to $A_{n}\na\Psi$ we get
$$\Curl(A_{n}\na\Psi)=\mat L_{\na\Psi}(\nah\axl A_{n})$$
with invertible $L_{\na\Psi}\in\Woqom$ satisfying
$L_{\na\Psi}\inv\in\Woqom$
by assumption on the regularity of $\Psi$.
Pick $\Theta\in\Cicomrtt$.
Then $L_{\na\Psi}^{-\top}\vect\Theta\in\Woqcom$ and 
since $A\na\Psi=\na\Phi\in\Lpomrtt$ with $\Curl(A\na\Psi)=0$ we have
$$\scpLtom{A_{n}\na\Psi}{\Curl(\mat L_{\na\Psi}^{-\top}\vect\Theta)}
\to\scpLtom{A\na\Psi}{\Curl(\mat L_{\na\Psi}^{-\top}\vect\Theta)}=0.$$
On the other hand we have for the left hand side
\begin{align*}
\scpLtom{A_{n}\na\Psi}{\Curl(\mat L_{\na\Psi}^{-\top}\vect\Theta)}
&=\scpLtom{\Curl(A_{n}\na\Psi)}{\mat L_{\na\Psi}^{-\top}\vect\Theta}\\
&=\scpLtom{L_{\na\Psi}(\nah\axl A_{n})}{L_{\na\Psi}^{-\top}\vect\Theta}\\
&=\scpLtom{\nah\axl A_{n}}{L_{\na\Psi}^{\top}L_{\na\Psi}^{-\top}\vect\Theta}\\
&=\scpLtom{\na\axl A_{n}}{\Theta}\\
&=\scpLtom{\axl A_{n}}{\Div\Theta}
\to\scpLtom{\axl A}{\Div\Theta}.
\end{align*}
Hence, $\na\axl A=0$ and therefore $A\in\so(3)$ is constant.
Thus, $\na(\Phi-A\Psi)=\na\Phi-A\na\Psi=0$ and
$\Phi=A\Psi+a$ with some $a\in\rt$.
\hfill$\square$

\appendix

\section{Appendix}

We show \eqref{eq:curlformel_allg} for smooth matrix fields
$X=[x_{nm}]_{n.m=1,2,3}$ and $Y=[y_{nm}]_{n.m=1,2,3}$.
The $l$-th row of $XY$ is the transpose of the vector
having the entries $x_{ln}y_{nk}$ for $k=1,2,3$.
Thus, the $l$-th row of $\Curl(XY)$ is the transpose of the vector
having the entries $\p_{i}(x_{ln}y_{nj})-\p_{j}(x_{ln}y_{ni})$ 
for $k=1,2,3$, where the $\curl$ of a vector field $v$ is written as
$$\curl v=
\begin{bmatrix}
\p_{2}v_{3}-\p_{3}v_{2}\\
\p_{3}v_{1}-\p_{1}v_{3}\\
\p_{1}v_{2}-\p_{2}v_{1}
\end{bmatrix}=[\p_{i}v_{j}-\p_{j}v_{i}]_{k=1,2,3}.$$
Therefore,
\begin{align*}
[\Curl(XY)]_{lk}
&=\p_{i}x_{ln}y_{nj}-\p_{j}x_{ln}y_{ni}+x_{ln}(\ub{\p_{i}y_{nj}-\p_{j}y_{ni}}_{=[\Curl Y]_{nk}})\\
&=\p_{i}x_{ln}y_{nj}-\p_{j}x_{ln}y_{ni}+[X\Curl Y]_{lk}.
\end{align*}
With the transpose of the $n$-th row of $Y$ denoted by $[y_{n}]_{j}:=y_{nj}$, we get
\begin{align*}
[\Curl(XY)]_{lk}-[X\Curl Y]_{lk}
&=[\na x_{ln}\times y_{n}]_{k}
\end{align*}
and hence for the $l$-th row 
$[\Curl(XY)-X\Curl Y]_{l}=[(\na x_{ln}\times y_{n})^{\top}]_{l}$.
Finally we obtain:
\begin{align*}
&\qquad\Curl(XY)-X\Curl Y
=\begin{bmatrix}
(\na x_{1n}\times y_{n})^{\top}\\
(\na x_{2n}\times y_{n})^{\top}\\
(\na x_{3n}\times y_{n})^{\top}
\end{bmatrix}\\
&=\begin{bmatrix}
(\na x_{11}\times y_{1})^{\top}\\
(\na x_{22}\times y_{2})^{\top}\\
(\na x_{33}\times y_{3})^{\top}
\end{bmatrix}
+\begin{bmatrix}
(\na x_{12}\times y_{2})^{\top}+(\na x_{13}\times y_{3})^{\top}\\
(\na x_{23}\times y_{3})^{\top}\\
0
\end{bmatrix}
+\begin{bmatrix}
0\\
(\na x_{21}\times y_{1})^{\top}\\
(\na x_{31}\times y_{1})^{\top}+(\na x_{32}\times y_{2})^{\top}
\end{bmatrix}\\
&=-\begin{bmatrix}
(\axl\inv y_{1}\na x_{11})^{\top}\\
(\axl\inv y_{2}\na x_{22})^{\top}\\
(\axl\inv y_{3}\na x_{33})^{\top}
\end{bmatrix}
-\begin{bmatrix}
(\axl\inv y_{2}\na x_{12})^{\top}+(\axl\inv y_{3}\na x_{13})^{\top}\\
(\axl\inv y_{3}\na x_{23})^{\top}\\
0
\end{bmatrix}\\
&\qquad
-\begin{bmatrix}
0\\
(\axl\inv y_{1}\na x_{21})^{\top}\\
(\axl\inv y_{1}\na x_{31})^{\top}+(\axl\inv y_{2}\na x_{32})^{\top}
\end{bmatrix}\\
&=-\mat\begin{bmatrix}
\axl\inv y_{1}&0&0\\
0&\axl\inv y_{2}&0\\
0&0&\axl\inv y_{3}
\end{bmatrix}
\begin{bmatrix}
\na x_{11}\\
\na x_{22}\\
\na x_{33}
\end{bmatrix}\\
&\qquad-\mat\begin{bmatrix}
0&\axl\inv y_{3}&-\axl\inv y_{2}\\
-\axl\inv y_{3}&0&0\\
0&0&0
\end{bmatrix}
\begin{bmatrix}
-\na x_{23}\\
\na x_{13}\\
-\na x_{12}
\end{bmatrix}\\
&\qquad-\mat\begin{bmatrix}
0&0&0\\
0&0&\axl\inv y_{1}\\
\axl\inv y_{2}&-\axl\inv y_{1}&0
\end{bmatrix}
\begin{bmatrix}
\na x_{32}\\
-\na x_{31}\\
\na x_{21}
\end{bmatrix}\\
&=\mat(L_{\diag,Y}\nah\dvec X+L_{\skew,Y}\nah\skewvec X+L_{\sym,Y}\nah\symvec X)
\end{align*}

\bibliographystyle{amsplain} 
\bibliography{unicontfo}

\end{document}